\documentclass[10pt, one side,emlines]{amsart}
\usepackage{amssymb,latexsym,xy,eucal,mathrsfs}
\textwidth=15.2cm \textheight=24cm \theoremstyle{plain}
\newtheorem{lem}{Lemma}[section]

\newtheorem{thm}[lem]{Theorem}

\theoremstyle{definition}

\newtheorem{obs}[lem]{Observation}
\newtheorem{defn}[lem]{Definition}

\begin{document}

\baselineskip 14truept
\title{Connectivity and edge-bipancyclicity of Hamming shell.}

\author{S. A. Mane and B. N. Waphare }

\dedicatory{Center for Advanced Studies in Mathematics,
	Department of Mathematics,\\ Savitribai Phule Pune University, Pune-411007, India.\\
manesmruti@yahoo.com : waphare@yahoo.com \\}

  \maketitle

 \begin{abstract} An
 Any graph obtained by deleting a Hamming code of length $n= 2^r - 1$ from a $n-$cube $Q_n$ is called as a Hamming shell.
 It is well known that a Hamming shell is vertex-transitive, edge-transitive, distance preserving \cite{bd,d1,d2,d3}. Moreover, it is Hamiltonian\cite{gs} and connected\cite{ddgz} . In this paper, we prove that a Hamming shell is edge-bipancyclic and $(n-1)-$connected.

 \end{abstract}



\section{Introduction} The hypercube $Q_n$ of dimension $n$ is one of the most versatile and powerful interconnection networks. The graph of $Q_n$ has many excellent properties such as recursive structure, regularity, symmetry, small diameter, low degree etc.

 A binary linear code of length $n = 2^r - 1( r \geq 2)$, with
parity-check matrix $H$ whose columns consist of all the nonzero vectors of $F^r_2$
is called a binary Hamming code of length $2^r - 1$. Binary Hamming codes are perfect codes (perfect independent dominating sets). Any graph obtained by deleting a Hamming code of length $n= 2^r - 1$ from a $n-$cube $Q_n$ is called as a Hamming shell.

  The study of the effect of removal of a dominating set on a network is interesting in the
 sense that the damage that the network experiences is more when a core of the network
 is removed.
 In interconnection networks, it means that when a dominating set is removed, the network
 suffers a heavy set back when the remaining part of the network gets scattered.  This gives
 the motivation to study the vulnerability of a network when dominating sets are removed.
 For  example,  in  a  network  with  a  node  having  connections  with  all  other  nodes  which
 are independent, the removal of this master node will result in complete chaos.  On the
 other hand,  when every node has connection with every other node in a network,  then
 the removal of any number of vertices will not affect the cohesiveness of the remaining
 part.  But, such a network is costly.  We have to examine various
 networks  which  are  economic  and  at  the  same  time  have  a  high  ability  to  withstand
 attacks  on  dominating  sets.   
 
 Connectivity is often
 used as a measure of system reliability and fault tolerance. 
 
 Cycles the  most  fundamental  classes  of  network
 topologies for parallel and distributed computing. Cycles
 are  suitable  for  designing  simple  algorithms  with  low  communication costs. Numerous efficient parallel
 algorithms  designed  on  cycles  for  solving  various  algebraic
 problems, graph  problems,  and  some  parallel  applications,
 such  as  those  in  image  and  signal  processing. These algorithms can also be used as control/data
 flow structures for distributed computing in arbitrary networks
 if  these  networks  can  embed  cycles  so  that  the  algorithms
 designed on cycles can be executed on the embedded cycles.
 In addition, if cycles of various lengths can be embedded, the
 number of simulated processors can be adjusted to increase the
 elasticity  of demand. As a result, the problem of embedding
 cycles of various lengths into a topology is crucial for network
 simulation. In this paper, we tried to  measure
 the fault tolerant bipancyclicity and connectivity  of a hypercube  when its dominating sets are under attack.
 
 The Dejter graph is obtained by deleting a copy of the Hamming code of length $7$ from the binary $7-$cube.
 The Dejter graph, and by extension any graph obtained by deleting a Hamming code of length $2^r - 1$ from a $(2^r - 1)-$cube is well known Hamming Shell which is vertex-transitive, edge-transitive, distance preserving \cite{bd,d1,d2,d3}. Also, Gregor and Skrekovski\cite{gs} proved that Hamming Shell is Hamiltonian. Moreover, Duckworth et al.\cite{ddgz} proved a lemma, that Hamming Shell is connected. Thus, a Hamming shell preserves many useful properties of hypercube.
 
 Inspiring from above applications and results, in this paper, we prove that a Hamming shell preserves two more useful properties of hypercube those are edge-bipancyclicity and $(n-1)-$connectivity.

Weichsel\cite{w} considered perfect dominating sets in hypercubes
as a way of generating distance preserving regular subgraphs. In this paper, we prove the existence of perfect dominating set $D$ of $Q_n$($n \geq 3$) such that $Q_n - D$ is
$(n-1)-$regular, $(n-1)-$connected and bipancyclic.

 For undefined terminology and
notations see West\cite{we} and Ling and Xing\cite{li}.

 \section {Preliminaries}
 
 An $n-$dimensional hypercube $Q_n$ can be represented as an undirected graph $Q_n = (V,E)$ such that $V$ consists of $2^n$
  nodes which are labeled as binary numbers of length $n$. $E$ is the set of edges that connect two nodes if and only if 
  they  differ in exactly one bit of their labels. 
The parity of a vertex in $Q_n$ is the parity of the number
of $1$s in its name, even or odd. As a number of even parity
vertices is same as a number of odd parity vertices in $Q_n$ and
hence it is balanced.
 $Q_n$ has many attractive properties, such as being bipartite, $n-$regular, $n-$connected.
 It is a bipancyclic graph in the sense that, for every even integer k, $4\leq k \leq V(Q_n)$, there exists a cycle $C_k$ of length $k$ in $Q_n$. Moreover, it is an edge-bipancyclic graph, as for any edge in $Q_n$, there is a cycle $C_k$ of every even length $k$, $4\leq k \leq V(Q_n)$, traversing through this edge.  
 
 $Q_{n+1}$ can be decomposed into two copies of $Q_{n}$ (denoted by $(Q_{n}, {0})$ and $(Q_{n}, {1}$) ) whose vertices are joined by $2^{n}$ edges of a perfect matching $R$. These edges are called cross edges. Let $V(Q_n) =\{v_i: 1 \leq i \leq 2^n \}$ and
 without loss $V(Q_{n+1}) = \{(v_i, 0), (v_i, 1): v_i \in V(Q_n), 1
 \leq i \leq 2^n\}$. So, we write $Q_{n+1} =
(Q_{n}, {0})\cup (Q_{n}, {1})\cup R$ where $ V(Q_{n}, {0})  =\{(v_i, 0): v_i \in V(Q_n), 1
 \leq i \leq 2^n\}$, $ V(Q_{n}, {1})  =\{(v_i, 1): v_i \in
 V(Q_n), 1 \leq i \leq 2^n\}$. Note
 that the end vertices of
 any edge in $R$ are called corresponding vertices of each
 other. 
  
  For $n = 2^m - 1$ and $m \geq 2$ if we denote by $k =
  2^{m+1} -1$ obviously $k = n + (n+1)$. We decompose $Q_k = Q_n
  \times Q_{n+1}$. Now for any $t \in V(Q_{n+1})$, we denote by
  $(Q_{n}, {t})$ the subgraph of $Q_k$ induced by the vertices whose
  last ${n+1}$ components form the tuple $t$ and $(D, t)$ denotes
  the subgraph of $(Q_{n}, {t})$. It is easy to observe that $(Q_{n}, {t})$ is isomorphic to $Q_n$.
  
   $F \subset V(Q_n)$ is said to be distant (strongly independent)
   set of vertices if $N[u]\cap N[v] = \phi$ for every $u, v \in F$.

\section {\bf construction of a perfect independent
dominating (PID) set of $Q_{2^{m+1} - 1}$ from PID set of
$Q_{2^{m} - 1}$, $m \geq 2$.}

 In this section, we prove a lemma which provides a technique of
constructing a PID set of $Q_{2^{m+1} -
1}$ from PID set of $Q_{2^{m} - 1}$, $m \geq 2$. We consider
$Q_n$ as a vector space over
$Z_2$. \\

 Whenever $n$ is of the form $n = 2^m -
1$($m \geq 2$), there exists a strongly independent subset of
$Q_n$ which is a dominating set and hence a perfect independent
dominating set. In fact, if $H$ is an $m \times (2^m - 1)$-matrix
obtained by taking all non-zero $m$-tuples of $Q_m$ as columns of
$H$ then the kernel of the linear transformation $H : Q_n
\longrightarrow Q_m$ is a perfect independent dominating set of
$Q_n$ (see \cite{h}). It is clear that any linear subspace which
is a perfect independent dominating set is a kernel of such type
of matrix (see \cite{d3}).\\

\begin{lem}\label{l9} If $D'$ is PID linear subspace of $Q_n$ $(n = 2^m - 1)$ then $D$ is PID set which is also a linear subspace of $Q_{n+n+1}$ where $D = \{ (x, y, j) :  x \in V(Q_{n}),  y\in ({D'+x}), \sum y_i =  j  \} $

\end{lem}

\begin{proof} Let
	
$$\left[\begin{array}{ccc}
H' & H'& 0\\
0 & 1 & 1
\end{array}\right ]$$


Now $H(x,y,j)^T=0$ \\
iff $\sum y_i =  j $ and $H'x + H'y = 0$.\\
iff $\sum y_i =  j $ and $x + y \in D'$.\\
iff $\sum y_i +  j = 0 $ and $ x \in V(Q_{n})$ and  $y\in ({D'+x})$.
Thus for any $ x \in V(Q_{n})$ choose $y\in ({D'+x})$ and choose $j = \sum  y_i$ to form $D$. 

\end{proof}

\begin{obs}\label{ob1} For $m \geq 2$ denote by $n = 2^m - 1$ and $k = 2^{m+1} - 1 = n + n + 1$. Let $Q_k = Q_n
	\times Q_{n+1}$. By Lemma \ref{l9}, It is easy to observe that\\
	 $[a]$ for any linear PID set $D'_0$ of $Q_n$,
	 $D'_i = D'_0 + e_i$ is also a PID set for
	 every $i$  ($1 \leq i \leq n$)  and further $V(Q_n) =
	 \bigcup^n_{i=0} D'_i$, $D'_i \bigcap D'_j = \phi$ for $i \neq j$ ($0
	 \leq i, j \leq n$). For every $i$ ($0 \leq i \leq n$), $D'_i$ is balanced\\
 $[b]$ for odd parity vertex $t \in V(Q_{n+1})$, $D \bigcap V(Q_{n}, {t}) =
\phi$ and for even parity vertex $t \in V(Q_{n+1})$, $D \bigcap
V(Q_{n}, {t}) = (D'_i, t)$ ($0 \leq i \leq n$).\\

\end{obs}

\section {\bf Edge-bipancyclicity of hypercubes with faulty vertices in PID set.}

To prove our main result of this section we need
following lemmas. 

\begin{lem}[\cite{xdx}\label{lm1}] Any two edges in $Q_n$ ($n \geq 2$) are included in a Hamiltonian cycle.\end{lem}

\begin{lem}\label{lm2}For any PID set $D$ of
$Q_3$, $Q_3 - D$ is Hamiltonian.
\end{lem}
\begin{proof} see Figure $1$.~\\
\begin{center}
\unitlength 1mm 
\linethickness{0.4pt}
\ifx\plotpoint\undefined\newsavebox{\plotpoint}\fi 
\begin{picture}(60.25,68.25)(0,0)
\put(10,26){\circle{1}} \put(10,54.75){\circle{1}}
\put(43.5,26){\circle{1}} \put(43.5,54.75){\circle{1}}
\put(10,39){\circle{1}} \put(10,67.75){\circle{1}}
\put(43.5,39){\circle{1}} \put(43.5,67.75){\circle{1}}
\put(4.75,15.75){\circle{1}} \put(4.75,44.5){\circle{1}}
\put(38.25,15.75){\circle{1}} \put(38.25,44.5){\circle{1}}
\put(4.75,28.75){\circle{1}} \put(4.75,57.5){\circle{1}}
\put(38.25,28.75){\circle{1}} \put(38.25,57.5){\circle{1}}
\put(26.25,26){\circle{1}} \put(26.25,54.75){\circle{1}}
\put(59.75,26){\circle{1}} \put(59.75,54.75){\circle{1}}
\put(26.25,39){\circle{1}} \put(26.25,67.75){\circle{1}}
\put(59.75,39){\circle{1}} \put(59.75,67.75){\circle{1}}
\put(21,15.75){\circle{1}} \put(21,44.5){\circle{1}}
\put(54.5,15.75){\circle{1}} \put(54.5,44.5){\circle{1}}
\put(21,28.75){\circle{1}} \put(21,57.5){\circle{1}}
\put(54.5,28.75){\circle{1}} \put(54.5,57.5){\circle{1}}
\put(9.75,55){\line(1,0){16.5}} \put(43.25,55){\line(1,0){16.5}}
\put(9.75,39.25){\line(1,0){16.5}}
\put(43.25,39.25){\line(1,0){16.5}} \put(4.5,16){\line(1,0){16.5}}
\put(38,16){\line(1,0){16.5}} \put(4.5,57.75){\line(1,0){16.5}}
\put(38,57.75){\line(1,0){16.5}} \put(9.75,26){\line(-1,-2){5.25}}
\put(9.75,54.75){\line(-1,-2){5.25}}
\put(43.25,39){\line(-1,-2){5.25}}
\put(43.25,67.75){\line(-1,-2){5.25}}
\multiput(59.75,26)(-.033557047,-.070469799){149}{\line(0,-1){.070469799}}
\multiput(59.75,54.75)(-.033557047,-.070469799){149}{\line(0,-1){.070469799}}
\multiput(26.25,39)(-.033557047,-.070469799){149}{\line(0,-1){.070469799}}
\multiput(26.25,67.75)(-.033557047,-.070469799){149}{\line(0,-1){.070469799}}
\put(10,39.25){\line(0,-1){13.25}}
\put(43.5,68){\line(0,-1){13.25}}
\put(4.75,57.75){\line(0,-1){13}} \put(38.25,29){\line(0,-1){13}}
\put(21,29){\line(0,-1){13}} \put(54.5,57.75){\line(0,-1){13}}
\put(26.25,68.25){\line(0,-1){13.5}}
\put(59.75,39.5){\line(0,-1){13.5}}
\multiput(4.93,57.93)(-.03125,-.046875){8}{\line(0,-1){.046875}}
\multiput(9.68,67.93)(-.033654,-.065705){12}{\line(0,-1){.065705}}
\multiput(8.872,66.353)(-.033654,-.065705){12}{\line(0,-1){.065705}}
\multiput(8.064,64.776)(-.033654,-.065705){12}{\line(0,-1){.065705}}
\multiput(7.257,63.199)(-.033654,-.065705){12}{\line(0,-1){.065705}}
\multiput(6.449,61.622)(-.033654,-.065705){12}{\line(0,-1){.065705}}
\multiput(5.641,60.045)(-.033654,-.065705){12}{\line(0,-1){.065705}}
\multiput(4.834,58.468)(-.033654,-.065705){12}{\line(0,-1){.065705}}
\put(4.43,57.68){\line(-1,0){.125}}
\put(4.18,57.68){\line(0,1){0}}
\put(4.18,57.68){\line(0,1){0}}
\put(9.93,68.18){\line(1,0){.9559}}
\put(11.841,68.15){\line(1,0){.9559}}
\put(13.753,68.121){\line(1,0){.9559}}
\put(15.665,68.091){\line(1,0){.9559}}
\put(17.577,68.062){\line(1,0){.9559}}
\put(19.489,68.033){\line(1,0){.9559}}
\put(21.4,68.003){\line(1,0){.9559}}
\put(23.312,67.974){\line(1,0){.9559}}
\put(25.224,67.944){\line(1,0){.9559}}
\put(9.93,67.68){\line(0,-1){.9286}}
\put(9.93,65.823){\line(0,-1){.9286}}
\put(9.93,63.965){\line(0,-1){.9286}}
\put(9.93,62.108){\line(0,-1){.9286}}
\put(9.93,60.251){\line(0,-1){.9286}}
\put(9.93,58.394){\line(0,-1){.9286}}
\put(9.93,56.537){\line(0,-1){.9286}}
\put(20.93,57.43){\line(0,-1){.9615}}
\put(20.93,55.507){\line(0,-1){.9615}}
\put(20.93,53.584){\line(0,-1){.9615}}
\put(20.93,51.66){\line(0,-1){.9615}}
\put(20.93,49.737){\line(0,-1){.9615}}
\put(20.93,47.814){\line(0,-1){.9615}}
\put(20.93,45.891){\line(0,-1){.9615}}
\multiput(25.68,54.93)(-.032986,-.072917){12}{\line(0,-1){.072917}}
\multiput(24.888,53.18)(-.032986,-.072917){12}{\line(0,-1){.072917}}
\multiput(24.096,51.43)(-.032986,-.072917){12}{\line(0,-1){.072917}}
\multiput(23.305,49.68)(-.032986,-.072917){12}{\line(0,-1){.072917}}
\multiput(22.513,47.93)(-.032986,-.072917){12}{\line(0,-1){.072917}}
\multiput(21.721,46.18)(-.032986,-.072917){12}{\line(0,-1){.072917}}
\put(4.68,44.68){\line(1,0){.9559}}
\put(6.591,44.68){\line(1,0){.9559}}
\put(8.503,44.68){\line(1,0){.9559}}
\put(10.415,44.68){\line(1,0){.9559}}
\put(12.327,44.68){\line(1,0){.9559}}
\put(14.239,44.68){\line(1,0){.9559}}
\put(16.15,44.68){\line(1,0){.9559}}
\put(18.062,44.68){\line(1,0){.9559}}
\put(19.974,44.68){\line(1,0){.9559}}
\multiput(9.68,39.18)(-.033654,-.067308){12}{\line(0,-1){.067308}}
\multiput(8.872,37.564)(-.033654,-.067308){12}{\line(0,-1){.067308}}
\multiput(8.064,35.949)(-.033654,-.067308){12}{\line(0,-1){.067308}}
\multiput(7.257,34.334)(-.033654,-.067308){12}{\line(0,-1){.067308}}
\multiput(6.449,32.718)(-.033654,-.067308){12}{\line(0,-1){.067308}}
\multiput(5.641,31.103)(-.033654,-.067308){12}{\line(0,-1){.067308}}
\multiput(4.834,29.487)(-.033654,-.067308){12}{\line(0,-1){.067308}}
\put(4.43,28.68){\line(0,-1){.9286}}
\put(4.43,26.823){\line(0,-1){.9286}}
\put(4.43,24.965){\line(0,-1){.9286}}
\put(4.43,23.108){\line(0,-1){.9286}}
\put(4.43,21.251){\line(0,-1){.9286}}
\put(4.43,19.394){\line(0,-1){.9286}}
\put(4.43,17.537){\line(0,-1){.9286}}
\put(4.43,29.18){\line(1,0){.9559}}
\put(6.341,29.121){\line(1,0){.9559}}
\put(8.253,29.062){\line(1,0){.9559}}
\put(10.165,29.003){\line(1,0){.9559}}
\put(12.077,28.944){\line(1,0){.9559}}
\put(13.989,28.886){\line(1,0){.9559}}
\put(15.9,28.827){\line(1,0){.9559}}
\put(17.812,28.768){\line(1,0){.9559}}
\put(19.724,28.709){\line(1,0){.9559}}
\put(26.18,39.43){\line(0,-1){.9464}}
\put(26.144,37.537){\line(0,-1){.9464}}
\put(26.108,35.644){\line(0,-1){.9464}}
\put(26.073,33.751){\line(0,-1){.9464}}
\put(26.037,31.858){\line(0,-1){.9464}}
\put(26.001,29.965){\line(0,-1){.9464}}
\put(25.965,28.073){\line(0,-1){.9464}}
\multiput(25.93,26.18)(-.032051,-.06891){12}{\line(0,-1){.06891}}
\multiput(25.16,24.526)(-.032051,-.06891){12}{\line(0,-1){.06891}}
\multiput(24.391,22.872)(-.032051,-.06891){12}{\line(0,-1){.06891}}
\multiput(23.622,21.218)(-.032051,-.06891){12}{\line(0,-1){.06891}}
\multiput(22.853,19.564)(-.032051,-.06891){12}{\line(0,-1){.06891}}
\multiput(22.084,17.91)(-.032051,-.06891){12}{\line(0,-1){.06891}}
\multiput(21.314,16.257)(-.032051,-.06891){12}{\line(0,-1){.06891}}
\put(9.68,26.18){\line(1,0){.9559}}
\put(11.591,26.15){\line(1,0){.9559}}
\put(13.503,26.121){\line(1,0){.9559}}
\put(15.415,26.091){\line(1,0){.9559}}
\put(17.327,26.062){\line(1,0){.9559}}
\put(19.239,26.033){\line(1,0){.9559}}
\put(21.15,26.003){\line(1,0){.9559}}
\put(23.062,25.974){\line(1,0){.9559}}
\put(24.974,25.944){\line(1,0){.9559}}
\put(37.93,57.68){\line(0,-1){.9464}}
\put(37.93,55.787){\line(0,-1){.9464}}
\put(37.93,53.894){\line(0,-1){.9464}}
\put(37.93,52.001){\line(0,-1){.9464}}
\put(37.93,50.108){\line(0,-1){.9464}}
\put(37.93,48.215){\line(0,-1){.9464}}
\put(37.93,46.323){\line(0,-1){.9464}}
\multiput(43.18,54.68)(-.033654,-.064103){12}{\line(0,-1){.064103}}
\multiput(42.372,53.141)(-.033654,-.064103){12}{\line(0,-1){.064103}}
\multiput(41.564,51.603)(-.033654,-.064103){12}{\line(0,-1){.064103}}
\multiput(40.757,50.064)(-.033654,-.064103){12}{\line(0,-1){.064103}}
\multiput(39.949,48.526)(-.033654,-.064103){12}{\line(0,-1){.064103}}
\multiput(39.141,46.987)(-.033654,-.064103){12}{\line(0,-1){.064103}}
\multiput(38.334,45.449)(-.033654,-.064103){12}{\line(0,-1){.064103}}
\put(37.93,44.68){\line(1,0){.9559}}
\put(39.841,44.709){\line(1,0){.9559}}
\put(41.753,44.739){\line(1,0){.9559}}
\put(43.665,44.768){\line(1,0){.9559}}
\put(45.577,44.797){\line(1,0){.9559}}
\put(47.489,44.827){\line(1,0){.9559}}
\put(49.4,44.856){\line(1,0){.9559}}
\put(51.312,44.886){\line(1,0){.9559}}
\put(53.224,44.915){\line(1,0){.9559}}
\multiput(59.43,68.18)(-.033654,-.06891){12}{\line(0,-1){.06891}}
\multiput(58.622,66.526)(-.033654,-.06891){12}{\line(0,-1){.06891}}
\multiput(57.814,64.872)(-.033654,-.06891){12}{\line(0,-1){.06891}}
\multiput(57.007,63.218)(-.033654,-.06891){12}{\line(0,-1){.06891}}
\multiput(56.199,61.564)(-.033654,-.06891){12}{\line(0,-1){.06891}}
\multiput(55.391,59.91)(-.033654,-.06891){12}{\line(0,-1){.06891}}
\multiput(54.584,58.257)(-.033654,-.06891){12}{\line(0,-1){.06891}}
\put(59.68,67.93){\line(0,-1){.9286}}
\put(59.68,66.073){\line(0,-1){.9286}}
\put(59.68,64.215){\line(0,-1){.9286}}
\put(59.68,62.358){\line(0,-1){.9286}}
\put(59.68,60.501){\line(0,-1){.9286}}
\put(59.68,58.644){\line(0,-1){.9286}}
\put(59.68,56.787){\line(0,-1){.9286}}
\put(43.68,68.18){\line(1,0){.9412}}
\put(45.562,68.18){\line(1,0){.9412}}
\put(47.444,68.18){\line(1,0){.9412}}
\put(49.327,68.18){\line(1,0){.9412}}
\put(51.209,68.18){\line(1,0){.9412}}
\put(53.091,68.18){\line(1,0){.9412}}
\put(54.974,68.18){\line(1,0){.9412}}
\put(56.856,68.18){\line(1,0){.9412}}
\put(58.739,68.18){\line(1,0){.9412}}
\put(43.18,39.18){\line(0,-1){.9464}}
\put(43.215,37.287){\line(0,-1){.9464}}
\put(43.251,35.394){\line(0,-1){.9464}}
\put(43.287,33.501){\line(0,-1){.9464}}
\put(43.323,31.608){\line(0,-1){.9464}}
\put(43.358,29.715){\line(0,-1){.9464}}
\put(43.394,27.823){\line(0,-1){.9464}}
\multiput(43.43,26.18)(-.033654,-.067308){12}{\line(0,-1){.067308}}
\multiput(42.622,24.564)(-.033654,-.067308){12}{\line(0,-1){.067308}}
\multiput(41.814,22.949)(-.033654,-.067308){12}{\line(0,-1){.067308}}
\multiput(41.007,21.334)(-.033654,-.067308){12}{\line(0,-1){.067308}}
\multiput(40.199,19.718)(-.033654,-.067308){12}{\line(0,-1){.067308}}
\multiput(39.391,18.103)(-.033654,-.067308){12}{\line(0,-1){.067308}}
\multiput(38.584,16.487)(-.033654,-.067308){12}{\line(0,-1){.067308}}
\put(43.68,25.93){\line(1,0){.9412}}
\put(45.562,25.959){\line(1,0){.9412}}
\put(47.444,25.989){\line(1,0){.9412}}
\put(49.327,26.018){\line(1,0){.9412}}
\put(51.209,26.047){\line(1,0){.9412}}
\put(53.091,26.077){\line(1,0){.9412}}
\put(54.974,26.106){\line(1,0){.9412}}
\put(56.856,26.136){\line(1,0){.9412}}
\put(58.739,26.165){\line(1,0){.9412}}
\put(38.18,28.93){\line(1,0){.9559}}
\put(40.091,28.959){\line(1,0){.9559}}
\put(42.003,28.989){\line(1,0){.9559}}
\put(43.915,29.018){\line(1,0){.9559}}
\put(45.827,29.047){\line(1,0){.9559}}
\put(47.739,29.077){\line(1,0){.9559}}
\put(49.65,29.106){\line(1,0){.9559}}
\put(51.562,29.136){\line(1,0){.9559}}
\put(53.474,29.165){\line(1,0){.9559}}
\multiput(59.18,39.43)(-.033654,-.067308){12}{\line(0,-1){.067308}}
\multiput(58.372,37.814)(-.033654,-.067308){12}{\line(0,-1){.067308}}
\multiput(57.564,36.199)(-.033654,-.067308){12}{\line(0,-1){.067308}}
\multiput(56.757,34.584)(-.033654,-.067308){12}{\line(0,-1){.067308}}
\multiput(55.949,32.968)(-.033654,-.067308){12}{\line(0,-1){.067308}}
\multiput(55.141,31.353)(-.033654,-.067308){12}{\line(0,-1){.067308}}
\multiput(54.334,29.737)(-.033654,-.067308){12}{\line(0,-1){.067308}}
\put(54.18,28.93){\line(0,-1){.9286}}
\put(54.18,27.073){\line(0,-1){.9286}}
\put(54.18,25.215){\line(0,-1){.9286}}
\put(54.18,23.358){\line(0,-1){.9286}}
\put(54.18,21.501){\line(0,-1){.9286}}
\put(54.18,19.644){\line(0,-1){.9286}}
\put(54.18,17.787){\line(0,-1){.9286}}
\put(29,6.25){\makebox(0,0)[cc]{Figure $1$}}
\put(16,63.25){\makebox(0,0)[cc]{$G_1$}}
\put(49.75,62.5){\makebox(0,0)[cc]{$G_2$}}
\put(15.75,33.5){\makebox(0,0)[cc]{$G_3$}}
\put(49.75,33.5){\makebox(0,0)[cc]{$G_4$}}
\end{picture}
\end{center}
\end{proof}

\begin{lem}\label{lm3} Consider the subgraphs $(Q_{3}, {0})$ and $(Q_{3}, {1})$ of $Q_{4}$. Let $G'$ be a cycle of
$(Q_{3}, {0})$ with $|V(G')| = 6$ and $R'=$ set of cross edges between $(Q_{3}, {1})$ and $G'$. Consider $G = (Q_{3}, {1}) \bigcup R'
\bigcup G'$. Then $G$ is an edge-bipancyclic graph. \end{lem}
\begin{proof}
It is easy to observe that any two cycles of length greater equal
$6$ in $Q_3$ have at least $4$ vertices common, and since a degree
of each vertex is $3$ so at least two edges common.

For every edge $T \in E(G)$, we want a cycle of every even length
$l$($4 \leq l \leq |V(G)|$) passing through $T$.(see Figure $2$)

For edges $e', f',h'\in E(G')$ let us denote their corresponding
edges in $ E(Q_{3}, {1})$ by $e, f ,h$ respectively. Also, let us denote
$e_1, e_2, e_3,e_4 \in R'$ the edges out of which $e_1,e_2$ join
the end vertices of $e$ and $e'$,  $f_1,f_2$ join the end vertices
of $f$ and $f'$ and $h_1,h_2$ join the end vertices of $h$ and
$h'$\\
\textbf{Case} $1$: For an edge $e$ in $E(Q_{3}, {1})$. As $(Q_{3}, {1})$ is an
edge-bipancyclic graph therefore there exists a cycle of every
even length $l$ ($4 \leq l \leq |V(Q_3)|$) in $(Q_{3}, {1})$ hence in $G$
containing $e$. Let $C$ be a Hamiltonian cycle in $(Q_{3}, {1})$
containing an edge $e$. Now, other than $e$ there is at least one
edge say $f$ lies on cycle $C$ such that its corresponding edge
$f'$ is on cycle $G'$. Consider cycles $(C-f)+f_1+f_2+f'$ and
$(C-f)+f_1+f_2+(G'-f')$ which contain $e$ with their length $10$
and $14$ respectively. Lastly, let $C_6$ be a cycle of length $6$
in $(Q_{3}, {1})$ which contains an edge $e$ and let $h$ be any other
edge on this cycle which have its corresponding edge say $h'$ in
$G'$. Now we construct a cycle of length $12$ containing $e$ as
$(C_6-h)+h_1+h_2+(G'-h')$.\\
 \textbf{Case} $2$: Consider an edge
$e'$ in $E(G')$, let $f'\neq e'$ be an edge on $G'$( there are $5$
such edges). Consider $C_k$ be a cycle in $(Q_{3}, {1})$ of even length
$k$ ($4 \leq k \leq |V(Q_3)|$ containing $f$. Now we construct a
cycle containing $e'$ of even length $l = k + |V(G')|$ (as $4 \leq
k \leq |V(Q_3)|$ therefore $|V(G')| + 4 \leq l \leq |V(G')| +
|V(Q_3)|$) as $(G' - f') + f_1 + (C_k - f) + f_2$. For $l =
|V(G')| + 2$ consider a cycle $(G' - f') + f_1 + f + f_2$ and for
$l=4$ a cycle is $e'+e_1+e_2+e$ containing $e'$.\\
\textbf{Case} $3$: For an edge $f_1$ in $R'$. let one end vertex
of $f_1$ is common with an edge $f'$ in $G'$ then Consider $C_k$
be a cycle in $(Q_{3}, {1})$ of even length $k$ ($4 \leq k \leq |V(Q_3)|$
containing $f$. Now we construct a cycle containing $f_1$ of even
length $l = k + |V(G')|$ ( therefore $|V(G')| + 4 \leq l \leq
|V(G')| + |V(Q_3)|$) as $(G' - f') + f_1 + (C_k - f) + f_2$. For
$l = |V(G')| + 2$ consider a cycle $(G' - f') + f_1 + f + f_2$ and
for $l=4$ and $l=6$  cycles are $f'+f_1+f_2+f$ and $ f' + f_1 +
(C_4 - f) + f_2$ respectively, where $C_k$ is cycle of length
$4$~\\
\begin{center}
\unitlength 1mm 
\linethickness{0.4pt}
\ifx\plotpoint\undefined\newsavebox{\plotpoint}\fi 
\begin{picture}(86.618,83.75)(0,0)
\put(49.25,2.125){\makebox(0,0)[cc]{Figure $2$}}
\put(9.309,35.875){\circle*{1.118}}
\put(66.309,35.875){\circle*{1.118}}
\put(9.309,45.875){\circle*{1.118}}
\put(66.309,45.875){\circle*{1.118}}
\put(21.559,45.875){\circle*{1.118}}
\put(29.059,54.375){\circle*{1.118}}
\put(86.059,54.375){\circle*{1.118}}
\put(28.559,27.875){\circle*{1.118}}
\put(85.559,27.875){\circle*{1.118}}
\put(4.059,28.125){\circle*{1.118}}
\put(4.059,54.125){\circle*{1.118}}
\put(61.059,54.125){\circle*{1.118}}
\put(21.559,36.125){\circle*{1.118}}
\put(78.559,36.125){\circle*{1.118}}
\put(21.809,46.375){\line(1,0){44.5}}
\put(21.559,36.625){\line(1,0){44.75}}
\put(29.559,54.625){\line(1,0){31.25}}
\qbezier(4.059,54.375)(44.309,83.75)(85.559,54.625)
\qbezier(9.309,36.375)(45.559,7.25)(78.809,36.625)
\qbezier(3.559,28.375)(45.059,4.5)(85.559,28.125)
\put(9.309,45.875){\line(1,0){12.25}}
\put(21.559,45.625){\line(0,-1){9.5}}
\put(9.309,36.375){\line(1,0){12.5}}
\put(9.059,45.625){\line(0,-1){9.25}}
\put(3.809,54.625){\line(0,-1){26.25}}
\put(3.809,28.125){\line(1,0){24.5}}
\put(28.809,54.375){\line(0,-1){26.25}}
\put(3.809,54.875){\line(1,0){24.75}}
\put(66.059,46.125){\line(0,-1){9.5}}
\put(66.059,36.375){\line(1,0){12.25}}
\put(85.809,54.625){\line(0,-1){26.25}}
\put(61.059,54.375){\line(1,0){24.75}}
\multiput(4.059,54.375)(.033557047,-.053691275){149}{\line(0,-1){.053691275}}
\multiput(8.809,36.375)(-.033687943,-.053191489){141}{\line(0,-1){.053191489}}
\multiput(21.309,36.125)(.03372093,-.036046512){215}{\line(0,-1){.036046512}}
\multiput(28.809,54.875)(-.033632287,-.038116592){223}{\line(0,-1){.038116592}}
\multiput(60.809,54.125)(.033653846,-.048076923){156}{\line(0,-1){.048076923}}
\multiput(78.559,36.125)(.033653846,-.036057692){208}{\line(0,-1){.036057692}}
\put(15.809,41.375){\makebox(0,0)[cc]{$Q^1_3$}}
\put(73.059,42.375){\makebox(0,0)[cc]{$G'$}}
\put(22.5,40.625){\makebox(0,0)[cc]{$e$}}
\put(64.25,42.375){\makebox(0,0)[cc]{$e'$}}
\put(15.5,52.125){\makebox(0,0)[cc]{$f$}}
\put(72.75,51.875){\makebox(0,0)[cc]{$f'$}}
\put(43.75,43.375){\makebox(0,0)[cc]{$e_1$}}
\put(43,33.875){\makebox(0,0)[cc]{$e_2$}}
\end{picture}
\end{center}
\end{proof}

\textbf{Note}: In above lemma if we take $|V(G')|=4$ or
$|V(G')|=8$ we still get the same result by using the almost same
technique of proof.

\begin{lem}\label{lm4} Consider the subgraphs $(Q_{n}, {0})$ and $(Q_{n}, {1})$ of $Q_{n+1}$. Let $G'$ be a subgraph of
$(Q_{n}, {0})$ and $R'=$ set of cross edges between $(Q_{n}, {1})$ and $G'$. Consider $G =(Q_{n}, {1}) \bigcup R'
\bigcup G'$. Then $G$ is \\
$(a)$ a bipancyclic graph if $G'$ is a Hamiltonian graph for $n \geq 2$.\\
$(b)$ an edge-bipancyclic graph if $G'$ is a bipancyclic graph for $n
\geq 3$.\end{lem}
\begin{proof}
For edges $e', f'_k\in E(G')$ let us denote their corresponding
edges in $ E(Q_{n}, {1})$ by $e$ and $f_k$ respectively. Also,  let us
denote $e_1, e_2, e_3,e_4 \in R'$ the edges of which $e_1,e_2$
join the end vertices of $e$ and $e'$, and $e_3,e_4$ join the end
vertices of $f_k$ and $f'_k$.~\\

  \textbf{$(a)$} We want to prove that $G = (Q_{n}, {1}) \bigcup R' \bigcup G'$
contains a cycle of every even length $l$ for $4 \leq l \leq
{|V(Q_n)| + |V(G')|}$. As $(Q_{n}, {1})$ is a bipancyclic graph therefore
there exists a cycle of every even length $l$ ($4 \leq l \leq
|V(Q_n)|$) in $G$.

Consider $C'$ a Hamiltonian cycle in $G'$ containing an edge say
$e'$. As $(Q_{n}, {1})$ is an edge-bipancyclic graph, let $C_k$ be
   a cycle in $(Q_{n}, {1})$ of even length $k$
   ($4 \leq k \leq |V(Q_n)|$ containing $e$. Now we construct a
   cycle of even length $l = k + |V(G')|$ ( therefore $|V(G')| + 4 \leq l \leq |V(G')| +
   |V(Q_n)|$) as $(C' - e') + e_1 + (C_k - e) + e_2$. For $l = |V(G')| + 2$
   consider $(C' - e') + e_1 + e + e_2$ .~\\

   \textbf{$(b)$}  For every edge $T \in E(G)$, we want a cycle of
every even length $l$($4 \leq l \leq |V(G)|$) passing through $T$.\\
\textbf{Case} $1$: Let $T = e \in E((Q_{n}, {1}))$, as $(Q_{n}, {1})$ is
edge-bipancyclic graph there exists a cycle
   of even length $l$ ($4 \leq l \leq |V(Q_n)|$) passing through $e$.
   Now we want cycle of even length $l$ ($|V(Q_n)| + 2 \leq l \leq |V(Q_n)| + |V(G')|$) passing through
   $e$. As $G'$ is a bipancyclic subgraph of
$(Q_{n}, {0})$, let $C'_k$ be
   a cycle in $G'$ of even length $k$
   ($4 \leq k \leq |V(G')|$) passing through an edge say
   $f'_k$ (choose $f'_k$ such that $f'_k \neq e'$ and $f'_k$ is not adjacent to $e'$). By using Lemma \ref{lm1}, there always exists Hamiltonian
   cycle in $(Q_{n}, {1})$ say $C$ containing $e$ and $f_k$. Now we construct a cycle of even length
   $|V(Q_n)| + k$ containing $e$ as, $(C - f_k) + e_3 + (C'_k - f'_k) +
   e_4$, if we take
   $(C - f_k) + e_3 + f'_k +
   e_4$ then we get cycle of length $|V(Q_n)| + 2$ containing $e$. \\
\textbf{Case} $2$: Consider $T = e' \in E(G')$. Let $C_k$ denotes
a cycle in $(Q_{n}, {1})$ of even length $k$ ($4 \leq k \leq |V(Q_n)|$)
passing through $e$. Then we construct a cycle of even length $l =
k + 2$ ( therefore $6 \leq l \leq |V(Q_n)| + 2$) in $G$ containing
$e'$ as $e' + e_1 + (C_k - e) + e_2$ and for $l = 4$ we take $e' +
e_1 + e + e_2$.\\
As $G'$ is bipancyclic, for every even $k$ ($4 \leq k \leq |V(G')|
- 4$) there exists a cycle say $C'_k$ of length $k$ passing
through an edge say $f'_k$ (we choose $f'_k$ in such a way that it
is not adjacent to an edge $e'$ and $f'_k \neq e'$). By using
Lemma \ref{lm1}, consider a Hamiltonian cycle $C$ in $(Q_{n}, {1})$
containing $e$ and $f_k$. Now a cycle of even length $l = k +
|V(Q_n)| + 2$ ( therefore $|V(Q_n)| + 6 \leq l \leq {|V(Q_n)| +
|V(G')| - 2}$) containing $e'$ in $G$ we construct as $ ((C - e)-
f_k )+ e_1 + e_2  + ((C'_k -e')- f'_k)+ e_3 + e_4$ and for $l =
|V(Q_n)| + 4$, we take $e' + e_1 + ((C - e)- f_k ) + e_2 + e_3 +
f'_k + e_4$.\\
\textbf{Case} $3$: And if $T = e_1$ where $e_1 \in R'$. Let $e_1$
and $e_2$ are the edges joining end vertices of corresponding
edges $e \in
   E(Q_{n}, {1})$ and $e' \in E(G')$. Now to find cycles passing through $e_1$ is
   to same as to find cycles passing through the edge $e'$. This completes the proof. \end{proof}

   \begin{lem}\label{lm5}For $n \geq 3$, consider the subgraphs $(Q_{n}, {0})$ and $(Q_{n}, {1})$ of
$Q_{n+1}$. Let $G_i$
be an edge-bipancyclic subgraphs of $(Q_{n}, {i})$, for $i \in \{0,
1\}$, such that every Hamiltonian cycle in $G_0$ contains at least two edges whose
corresponding edges are in $G_{1}$ and vice-versa. The set of only those
cross edges which join both the end vertices of the corresponding
edges from $G_0$ and $G_1$ be denoted by $R'$. Consider $G = G_0
\bigcup R' \bigcup G_1$. Then $G$ is an edge-bipancyclic
graph.\end{lem}
\begin{proof}\textbf{Case} $1$: Let $e \in E(G_i)$, for $i \in \{0, 1\}$. Due to symmetry, it is sufficient to prove that any edge
$e \in E(G_1)$ is contained in a cycle of even length $l$, $4 \leq
l \leq {|V(G_1)| + |V(G_0)|}$. As $G_1$ is an edge-bipancyclic
graph, therefore, $e$ is contained in a cycle of even length $l$, $4
\leq l \leq {|V(G_1)|}$.

 Let $C$ denotes Hamiltonian cycle in $G_1$ containing $e$. By
 assumption, there exists at least one edge of $C$ other than $e$
 say $f$ such that its corresponding edge $f'$ is in $G_0$. Thus if $f = \langle A_1, A_2 \rangle$
 then $e_1 = \langle A_1, B_1 \rangle$ and $e_2 = \langle A_2, B_2 \rangle$ such that $e_1, e_2 \in R'$
 where $A_1, A_2 \in V(G_1)$ and $B_1, B_2 \in V(G_0)$. Now let us denote by
 $C'_k$ a cycle of even length $k$ ($4 \leq k \leq |V(G_0)|$) containing $f'$ in
 $G_0$, Now we construct cycle of even length $l = k + |V(G_1)|$ (as $4 \leq k \leq |V(G_0)|$, therefore $|C| + 4 \leq l \leq {{|V(G_1)| + |V(G_0)|}}$)
 containing $e$
 as $(C - f) + e_1 + (C'_k - f') + e_2$ and for $l = |C| + 2$ we
 take cycle
 $(C - f) + e_1 + f' + e_2$.\\
 \textbf{Case} $2$: For $e_1 \in R'$. By assumption there exists $e_2 \in R'$, such that $e_1$ and $e_2$ are the edges joining end vertices of
  corresponding edges $f \in
   E(G_1)$ and $f' \in E(G_0)$. Let $C_k$ denotes a cycle of even length $k$ ($4 \leq k \leq |V(G_1)|$) containing $f$ in
 $G_1$. Now we construct a cycle of even length $l = k + 2$ ( therefore $6 \leq l \leq {|V(G_1)| + 2}$) containing $e_1$
 as $(C_k - f) + e_1 + f' + e_2$ and for $l = 4$ we take
 $f + e_1 + f' + e_2$. Let $C$ be the Hamiltonian cycle containing an edge $f$
 in $G_1$ and let $C'_k$ be denotes cycle of length $k$ ($4 \leq k
 \leq |V(G_0)|$ containing $f'$ in $G_0$. Now we construct cycle of even length $l = |V(G_1)| + k$
 ( therefore $|V(G_1)| + 4 \leq l \leq {{|V(G_1)| + |V(G_0)|}}$) containing $e_1$
 as $(C - f) + e_1 + (C'_k - f') + e_2$.
\end{proof}

\begin{lem}\label{lm6} For $n \geq 3$, consider the subgraphs $(Q_{n}, {0})$ and $(Q_{n}, {1})$ of
$Q_{n+1}$. Let $G_i$ be the subgraph of $(Q_{n}, {i})$, for $i \in
\{0, 1\}$. If an even length cycle $C$ in $Q_{n+1}$ contains at
least $3$ vertices of $G_i$ for all $i \in \{0, 1\}$, then $|E(C)|
\bigcap |E(G_i)| \geq 2$ for all $i \in \{0, 1\}$.
\end{lem}
\begin{proof} For every vertex in $(Q_{n}, {0})$ there exists a unique vertex in
$(Q_{n}, {1})$ which is adjacent to it, hence the proof follows
immediately. \end{proof}

\begin{lem}\label{lm7} For $n \geq 3$, let $F$ denotes a strongly
   independent set and $P$ denotes a spanning subgraph of $Q_n$ such that for every vertex $v$ of $Q_n$,
   $d_p(v) = 1$ or $2$. Then $|E(Q_n - F)| \bigcap |E(P)| \geq 2$. \end{lem}
\begin{proof} As $P$ is a spanning subgraph, every vertex $u \in V(F)$ lies on $P$. If $|F| = 1$, then we are through.
Now let $|F| \geq 2$. Let $u, v \in F$, and let us denote by $u_i$
and $v_i$ the $n$ neighbors of $u$ and $v$ in $Q_n$ respectively
($1 \leq i \leq n$). Let $uu_1, uu_2 \in E(P)$. Due to
trianglefreeness of $Q_n$, $u_i$ ($3 \leq i \leq n$) can not be
adjacent to $u_1$ and $u_2$, but $d_p(u_i) = 1$ or $2$, therefore,
$u_i$ ($3 \leq i \leq n$) is adjacent to some vertex $w_i$ in $P$.
We note that $w_i \notin F$ otherwise it will contradict strongly
independentness of $F$. By similar arguments, $v_i$ ($3 \leq i \leq
n$) is adjacent to some vertex $w'_i$ in $P$ which is not in $F$.
Thus, we proved.\end{proof}

For any $x \in Q_n$ the mapping $u \longrightarrow {u+x}$ is a
graph automorphism of $Q_n$ also for any permutation $\sigma \in
S_n$, the mapping $(v_1, v_2,...,v_n) \longrightarrow
(v_{\sigma{1}}, v_{\sigma{2}}, v_{\sigma{3}},....,v_{\sigma{n}})$
is also a graph automorphism of $Q_n$, therefore the following
result is true in view of Lemma \ref{l9}.

\begin{lem}\label{lm8}If $D$ is a PID set of
$Q_n$ such that $Q_n - D$ is edge-bipancyclic then;\\
$(a)$ $Q_n - D_i$ is also edge-bipancyclic where for $1 \leq i
\leq n$, $D_i = D + e_i$ and $e_i = (0, 0,....,1,0...,0)$ ($1$ is
at $i^{th}$ position);\\
$(b)$ in addition, if $D$ is linear then for any linear perfect
independent dominating set $D_0$ of $Q_n$, $Q_n - D_0$ is also
edge-bipancyclic.
\end{lem}

Now we prove our main result of this section which gives
edge-bipancyclicity of a hamming shell.

\begin{thm}\label{tm2}For any coset of any linear PID set $D$ of
$Q_n$ $(n = 2^m - 1 $ and $m \geq 3)$, $Q_n - D$ is
edge-bipancyclic.\end{thm}
\begin{proof}
We follow notations of Observation \ref{ob1}.\\

\textbf{Part I} : We decompose $Q_7$ as $Q_7 = Q_3 \times Q_4$.
Let $C = t_1t_2t_3.....t_{16}t_1$ be a Hamiltonian cycle in
$Q_{4}$. Without loss of generality say $t_j$ is of odd parity,
for odd integers $j$, $1 \leq j \leq 16$. Obviously, $t_j$ is of
even parity, for even integers $j$, $1 \leq j \leq 16$ (see Figure
$3$).~\\
\begin{center}
\unitlength 1mm 
\linethickness{0.4pt}
\ifx\plotpoint\undefined\newsavebox{\plotpoint}\fi 
\begin{picture}(102.258,68.5)(0,0)
\put(12,11){\circle{8.016}} \put(36.5,11){\circle{8.016}}
\put(60.75,11){\circle{8.016}} \put(98.25,11){\circle{8.016}}
\put(12,32.75){\circle{8.016}} \put(36.5,32.75){\circle{8.016}}
\put(60.75,32.75){\circle{8.016}}
\put(98.25,32.75){\circle{8.016}} \put(9.5,29.25){\line(0,-1){15}}
\put(34,29.25){\line(0,-1){15}} \put(58.25,29.25){\line(0,-1){15}}
\put(95.75,29.25){\line(0,-1){15}} \put(9.5,14.25){\line(0,1){0}}
\put(34,14.25){\line(0,1){0}} \put(58.25,14.25){\line(0,1){0}}
\put(95.75,14.25){\line(0,1){0}}
\put(11.75,28.25){\line(0,-1){13.75}}
\put(36.25,28.25){\line(0,-1){13.75}}
\put(60.5,28.25){\line(0,-1){13.75}}
\put(98,28.25){\line(0,-1){13.75}}
\put(14.5,29.5){\line(0,-1){15}} \put(39,29.5){\line(0,-1){15}}
\put(63.25,29.5){\line(0,-1){15}}
\put(100.75,29.5){\line(0,-1){15}}
\put(11.75,8.5){\makebox(0,0)[cc]{\tiny{$(Q_3, {t_1})$}}}
\put(94.25,68.5){\makebox(0,0)[cc]{}}
\put(94.25,68.5){\makebox(0,0)[cc]{}}
\put(94.25,68.5){\makebox(0,0)[cc]{}}
\put(94.25,68.5){\makebox(0,0)[cc]{}}
\put(12.25,32.75){\makebox(0,0)[cc]{\tiny{$(Q_3, {t_2})$}}}
\put(36.5,8.75){\makebox(0,0)[cc]{\tiny{$(Q_3, {t_3})$}}}
\put(36.5,32.75){\makebox(0,0)[cc]{\tiny{$(Q_3, {t_4})$}}}
\put(60.5,9){\makebox(0,0)[cc]{\tiny{$(Q_3, {t_5})$}}}
\put(61,33){\makebox(0,0)[cc]{\tiny{$(Q_3, {t_6})$}}}
\put(98,8.75){\makebox(0,0)[cc]{\tiny{$(Q_3, {t_{15}})$}}}
\put(98.5,33){\makebox(0,0)[cc]{\tiny{$(Q_3, {t_{16}})$}}}
\multiput(15.5,31.25)(.0336914063,-.0424804688){512}{\line(0,-1){.0424804688}}
\multiput(16,32.5)(.0336734694,-.0423469388){490}{\line(0,-1){.0423469388}}
\multiput(15.25,30.5)(-.03125,.03125){8}{\line(0,1){.03125}}
\multiput(15,35.5)(.0337078652,-.0411985019){534}{\line(0,-1){.0411985019}}
\multiput(39.75,30.25)(.0336914063,-.0415039063){512}{\line(0,-1){.0415039063}}
\multiput(40.25,31.75)(.0336734694,-.0418367347){490}{\line(0,-1){.0418367347}}
\multiput(40.25,34)(.0337022133,-.041750503){497}{\line(0,-1){.041750503}}
\put(66.68,32.68){\line(1,0){.97}}
\put(68.62,32.68){\line(1,0){.97}}
\put(70.56,32.68){\line(1,0){.97}}
\put(72.5,32.68){\line(1,0){.97}}
\put(74.44,32.68){\line(1,0){.97}}
\put(76.38,32.68){\line(1,0){.97}}
\put(78.32,32.68){\line(1,0){.97}}
\put(80.26,32.68){\line(1,0){.97}}
\put(82.2,32.68){\line(1,0){.97}}
\put(84.14,32.68){\line(1,0){.97}}
\put(86.08,32.68){\line(1,0){.97}}
\put(88.02,32.68){\line(1,0){.97}}
\put(89.96,32.68){\line(1,0){.97}}
\put(66.93,27.68){\line(1,0){.9896}}
\put(68.909,27.701){\line(1,0){.9896}}
\put(70.888,27.721){\line(1,0){.9896}}
\put(72.867,27.742){\line(1,0){.9896}}
\put(74.846,27.763){\line(1,0){.9896}}
\put(76.826,27.784){\line(1,0){.9896}}
\put(78.805,27.805){\line(1,0){.9896}}
\put(80.784,27.826){\line(1,0){.9896}}
\put(82.763,27.846){\line(1,0){.9896}}
\put(84.742,27.867){\line(1,0){.9896}}
\put(86.721,27.888){\line(1,0){.9896}}
\put(88.701,27.909){\line(1,0){.9896}}
\put(66.43,22.68){\line(1,0){.9688}}
\put(68.367,22.659){\line(1,0){.9688}}
\put(70.305,22.638){\line(1,0){.9688}}
\put(72.242,22.617){\line(1,0){.9688}}
\put(74.18,22.596){\line(1,0){.9688}}
\put(76.117,22.576){\line(1,0){.9688}}
\put(78.055,22.555){\line(1,0){.9688}}
\put(79.992,22.534){\line(1,0){.9688}}
\put(81.93,22.513){\line(1,0){.9688}}
\put(83.867,22.492){\line(1,0){.9688}}
\put(85.805,22.471){\line(1,0){.9688}}
\put(87.742,22.451){\line(1,0){.9688}}
\put(66.43,17.68){\line(1,0){.9674}}
\put(68.364,17.636){\line(1,0){.9674}}
\put(70.299,17.593){\line(1,0){.9674}}
\put(72.234,17.549){\line(1,0){.9674}}
\put(74.169,17.506){\line(1,0){.9674}}
\put(76.104,17.462){\line(1,0){.9674}}
\put(78.038,17.419){\line(1,0){.9674}}
\put(79.973,17.375){\line(1,0){.9674}}
\put(81.908,17.332){\line(1,0){.9674}}
\put(83.843,17.288){\line(1,0){.9674}}
\put(85.778,17.245){\line(1,0){.9674}}
\put(87.712,17.201){\line(1,0){.9674}}
\put(66.93,12.93){\line(1,0){.9688}}
\put(68.867,12.93){\line(1,0){.9688}}
\put(70.805,12.93){\line(1,0){.9688}}
\put(72.742,12.93){\line(1,0){.9688}}
\put(74.68,12.93){\line(1,0){.9688}}
\put(76.617,12.93){\line(1,0){.9688}}
\put(78.555,12.93){\line(1,0){.9688}}
\put(80.492,12.93){\line(1,0){.9688}}
\put(82.43,12.93){\line(1,0){.9688}}
\put(84.367,12.93){\line(1,0){.9688}}
\put(86.305,12.93){\line(1,0){.9688}}
\put(88.242,12.93){\line(1,0){.9688}}
\put(54,2.75){\makebox(0,0)[cc]{Figure $3$}}
\multiput(64.5,31.25)(.0336826347,-.0583832335){334}{\line(0,-1){.0583832335}}
\multiput(64.75,32.75)(.0336538462,-.0556318681){364}{\line(0,-1){.0556318681}}
\multiput(64.75,34.75)(.0336658354,-.0548628429){401}{\line(0,-1){.0548628429}}
\multiput(86.5,33)(.0337301587,-.0773809524){252}{\line(0,-1){.0773809524}}
\multiput(85.25,32.75)(.0337078652,-.0767790262){267}{\line(0,-1){.0767790262}}
\multiput(83.75,33)(.0337171053,-.0740131579){304}{\line(0,-1){.0740131579}}
\multiput(15.5,13)(.1420863309,.0337230216){556}{\line(1,0){.1420863309}}
\multiput(15.75,11.5)(.1387915937,.0337127846){571}{\line(1,0){.1387915937}}
\multiput(15,14)(.1425359712,.0337230216){556}{\line(1,0){.1425359712}}
\end{picture}
 \end{center}~\\

Let $D'_0$ be the Kernel of matrix $H'$
and let $D'_i = D'_0 + e_i$ (for $1 \leq i \leq 3$, $e_i =
(0,..1,..0)$ $1$ is at $i^{th}$ position) be its cosets. Now by
lemma \ref{l9} and observation \ref{ob1} , we get the kernel of $H$ say $D$ such that for
odd integer $j$, $D \cap (Q_3, {t_j}) = \phi$ and for even integer
$j$, $D \cap (Q_3, {t_j}) = (D', {t_j})$ where $D' \in
\{D'_0,\ldots,D'_3\}$ ($1 \leq j \leq 16$).\\

Now we claim that $Q_7 - D$ is edge-bipancyclic.\\

Let us denote by $G_j =(Q_3, {t_j})- ( D \cap (Q_3, {t_j}) )$ ($1
\leq j \leq 16$). For every even integer $j$ ($1 \leq j \leq 16$),
$G_j$ is isomorphic to one of the $G'_i$ ($G'_i = Q_3 - D'_i$, see
Lemma \ref{lm2}) which is a Hamiltonian cycle ($0 \leq i \leq 3$). And
for odd integer $j$, $G_j = (Q_3, {t_j})$ which is isomorphic to
$Q_3$.

Let us denote by $E_j$ the set of cross edges between $G_j$ and
$G_{j+1}$ ($1 \leq j \leq 15$), therefore $E_j = \{e :$ one end
vertex of an edge $e$ is in  $G_l, l$ is an even integer with $l
\in \{j, j+1\} \}$ (see Figure $4$).

Let $H_j = G_j \bigcup E_j \bigcup G_{j+1}$ ($j$ is an odd
integer, $1 \leq j \leq 16$), by Lemma \ref{lm3}, $H_j$ is an edge
bipancyclic graph (see Figure $5$).

As $V(Q_7 - D) = \bigcup^{16}_{j=1} V(G_j)$. (Also, it is well known
 that $Q_n$ for any $n \geq 1$ has a Hamiltonian path between every two vertices of opposite
 parity\cite{gs} so we can take some other Hamiltonian
 cycle in $Q_4$ instead of cycle $t_1t_2 \ldots t_{16}t_1$
 but due to bipartiteness of
hypercubes the cross edges $E_i$ will in between $G_j$ (for $j$
odd only) and $G_k$ (for $k$ even only) for some $1\leq j,k \leq
16$.) Hence, it is sufficient to prove that $H_1 \bigcup E_2
\bigcup H_3$ is edge-bipancyclic .

By using Lemma \ref{lm6}, we observe that every Hamiltonian cycle in
$H_1$ contains at least two edges of $G_2$. And by using Lemma
\ref{lm7}, we observe that every Hamiltonian cycle $C$ in $H_3$
contains at least two edges of $G_3$ which corresponds to the
edges of $G_2$.

Now by using Lemma \ref{lm5}, we observe that $H_1 \bigcup E_2 \bigcup
H_3$ is edge-bipancyclic.

Thus we proved the existence of linear perfect independent
dominating set of $Q_7$ such that
$Q_7 - D$ is edge-bipancyclic.\\

By Lemma \ref{lm8} it is sufficient what we proved that there exists a
linear perfect independent dominating set $D$ of $Q_7$ such that
$Q_7 - D$ is
edge-bipancyclic.~\\
\begin{center}
\unitlength 1mm 
\linethickness{0.4pt}
\ifx\plotpoint\undefined\newsavebox{\plotpoint}\fi 
\begin{picture}(101.008,61.25)(0,0)
\put(10.75,7.25){\circle{8.016}} \put(35.25,7.25){\circle{8.016}}
\put(59.5,7.25){\circle{8.016}} \put(97,7.25){\circle{8.016}}
\put(10.75,29){\circle{8.016}} \put(35.25,29){\circle{8.016}}
\put(59.5,29){\circle{8.016}} \put(97,29){\circle{8.016}}
\put(8.25,25.5){\line(0,-1){15}} \put(32.75,25.5){\line(0,-1){15}}
\put(57,25.5){\line(0,-1){15}} \put(94.5,25.5){\line(0,-1){15}}
\put(8.25,10.5){\line(0,1){0}} \put(32.75,10.5){\line(0,1){0}}
\put(57,10.5){\line(0,1){0}} \put(94.5,10.5){\line(0,1){0}}
\put(10.5,24.5){\line(0,-1){13.75}}
\put(35,24.5){\line(0,-1){13.75}}
\put(59.25,24.5){\line(0,-1){13.75}}
\put(96.75,24.5){\line(0,-1){13.75}}
\put(10.5,4.75){\makebox(0,0)[cc]{\tiny{$(Q_3, {t_1})$}}}
\put(94.25,61.25){\makebox(0,0)[cc]{}}
\put(94.25,61.25){\makebox(0,0)[cc]{}}
\put(94.25,61.25){\makebox(0,0)[cc]{}}
\put(94.25,61.25){\makebox(0,0)[cc]{}}
\put(35.25,5){\makebox(0,0)[cc]{\tiny{$(Q_3, {t_3})$}}}
\put(59.25,5.25){\makebox(0,0)[cc]{\tiny{$(Q_3, {t_5})$}}}
\put(96.75,5){\makebox(0,0)[cc]{\tiny{$(Q_3, {t_{15}})$}}}
\multiput(14.25,27.5)(.0336914063,-.0424804688){512}{\line(0,-1){.0424804688}}
\multiput(14.75,28.75)(.0336734694,-.0423469388){490}{\line(0,-1){.0423469388}}
\multiput(14,26.75)(-.03125,.03125){8}{\line(0,1){.03125}}
\multiput(38.5,26.5)(.0336914063,-.0415039063){512}{\line(0,-1){.0415039063}}
\multiput(39,28)(.0336734694,-.0418367347){490}{\line(0,-1){.0418367347}}
\put(65.43,28.93){\line(1,0){.97}}
\put(67.37,28.93){\line(1,0){.97}}
\put(69.31,28.93){\line(1,0){.97}}
\put(71.25,28.93){\line(1,0){.97}}
\put(73.19,28.93){\line(1,0){.97}}
\put(75.13,28.93){\line(1,0){.97}}
\put(77.07,28.93){\line(1,0){.97}}
\put(79.01,28.93){\line(1,0){.97}}
\put(80.95,28.93){\line(1,0){.97}}
\put(82.89,28.93){\line(1,0){.97}}
\put(84.83,28.93){\line(1,0){.97}}
\put(86.77,28.93){\line(1,0){.97}}
\put(88.71,28.93){\line(1,0){.97}}
\put(65.18,18.93){\line(1,0){.9688}}
\put(67.117,18.909){\line(1,0){.9688}}
\put(69.055,18.888){\line(1,0){.9688}}
\put(70.992,18.867){\line(1,0){.9688}}
\put(72.93,18.846){\line(1,0){.9688}}
\put(74.867,18.826){\line(1,0){.9688}}
\put(76.805,18.805){\line(1,0){.9688}}
\put(78.742,18.784){\line(1,0){.9688}}
\put(80.68,18.763){\line(1,0){.9688}}
\put(82.617,18.742){\line(1,0){.9688}}
\put(84.555,18.721){\line(1,0){.9688}}
\put(86.492,18.701){\line(1,0){.9688}}
\put(65.18,13.93){\line(1,0){.9674}}
\put(67.114,13.886){\line(1,0){.9674}}
\put(69.049,13.843){\line(1,0){.9674}}
\put(70.984,13.799){\line(1,0){.9674}}
\put(72.919,13.756){\line(1,0){.9674}}
\put(74.854,13.712){\line(1,0){.9674}}
\put(76.788,13.669){\line(1,0){.9674}}
\put(78.723,13.625){\line(1,0){.9674}}
\put(80.658,13.582){\line(1,0){.9674}}
\put(82.593,13.538){\line(1,0){.9674}}
\put(84.528,13.495){\line(1,0){.9674}}
\put(86.462,13.451){\line(1,0){.9674}}
\put(65.68,9.18){\line(1,0){.9688}}
\put(67.617,9.18){\line(1,0){.9688}}
\put(69.555,9.18){\line(1,0){.9688}}
\put(71.492,9.18){\line(1,0){.9688}}
\put(73.43,9.18){\line(1,0){.9688}}
\put(75.367,9.18){\line(1,0){.9688}}
\put(77.305,9.18){\line(1,0){.9688}}
\put(79.242,9.18){\line(1,0){.9688}}
\put(81.18,9.18){\line(1,0){.9688}}
\put(83.117,9.18){\line(1,0){.9688}}
\put(85.055,9.18){\line(1,0){.9688}}
\put(86.992,9.18){\line(1,0){.9688}}
\put(54,-6){\makebox(0,0)[cc]{Figure $4$}}
\multiput(63.25,27.5)(.0336826347,-.0583832335){334}{\line(0,-1){.0583832335}}
\multiput(63.5,29)(.0336538462,-.0556318681){364}{\line(0,-1){.0556318681}}
\multiput(84,29)(.0337078652,-.0767790262){267}{\line(0,-1){.0767790262}}
\multiput(82.5,29.25)(.0337171053,-.0740131579){304}{\line(0,-1){.0740131579}}
\multiput(14.25,9.25)(.1420863309,.0337230216){556}{\line(1,0){.1420863309}}
\multiput(14.5,7.75)(.1387915937,.0337127846){571}{\line(1,0){.1387915937}}
\put(10.5,29){\makebox(0,0)[cc]{\tiny{$G_2$}}}
\put(35,30.25){\makebox(0,0)[cc]{\tiny{$G_4$}}}
\put(59.75,29){\makebox(0,0)[cc]{\tiny{$G_6$}}}
\put(96.75,29){\makebox(0,0)[cc]{\tiny{$G_{16}$}}}
\put(11.75,19){\makebox(0,0)[cc]{\tiny{$E_1$}}}
\put(23.5,19.5){\makebox(0,0)[cc]{\tiny{$E_2$}}}
\put(36.25,19.75){\makebox(0,0)[cc]{\tiny{$E_3$}}}
\put(45,21.75){\makebox(0,0)[cc]{\tiny{$E_4$}}}
\put(75.25,25){\makebox(0,0)[cc]{\tiny{$E_{16}$}}}
\end{picture}
\end{center}~\\

\begin{center}
\unitlength 1mm 
\linethickness{0.4pt}
\ifx\plotpoint\undefined\newsavebox{\plotpoint}\fi 
\begin{picture}(109.5,61.25)(0,0)
\put(16.5,12.5){\circle{8.016}} \put(41,12.5){\circle{8.016}}
\put(65.25,12.5){\circle{8.016}} \put(102.75,12.5){\circle{8.016}}
\put(16.5,34.25){\circle{8.016}} \put(41,34.25){\circle{8.016}}
\put(65.25,34.25){\circle{8.016}}
\put(102.75,34.25){\circle{8.016}} \put(14,30.75){\line(0,-1){15}}
\put(38.5,30.75){\line(0,-1){15}}
\put(62.75,30.75){\line(0,-1){15}}
\put(100.25,30.75){\line(0,-1){15}} \put(14,15.75){\line(0,1){0}}
\put(38.5,15.75){\line(0,1){0}} \put(62.75,15.75){\line(0,1){0}}
\put(100.25,15.75){\line(0,1){0}}
\put(16.25,29.75){\line(0,-1){13.75}}
\put(40.75,29.75){\line(0,-1){13.75}}
\put(65,29.75){\line(0,-1){13.75}}
\put(102.5,29.75){\line(0,-1){13.75}}
\put(94.25,61.25){\makebox(0,0)[cc]{}}
\put(94.25,61.25){\makebox(0,0)[cc]{}}
\put(94.25,61.25){\makebox(0,0)[cc]{}}
\put(94.25,61.25){\makebox(0,0)[cc]{}}
\multiput(20,32.75)(.0336914063,-.0424804688){512}{\line(0,-1){.0424804688}}
\multiput(20.5,34)(.0336734694,-.0423469388){490}{\line(0,-1){.0423469388}}
\multiput(19.75,32)(-.03125,.03125){8}{\line(0,1){.03125}}
\multiput(44.25,31.75)(.0336914063,-.0415039063){512}{\line(0,-1){.0415039063}}
\multiput(44.75,33.25)(.0336734694,-.0418367347){490}{\line(0,-1){.0418367347}}
\put(71.18,34.18){\line(1,0){.97}}
\put(73.12,34.18){\line(1,0){.97}}
\put(75.06,34.18){\line(1,0){.97}} \put(77,34.18){\line(1,0){.97}}
\put(78.94,34.18){\line(1,0){.97}}
\put(80.88,34.18){\line(1,0){.97}}
\put(82.82,34.18){\line(1,0){.97}}
\put(84.76,34.18){\line(1,0){.97}}
\put(86.7,34.18){\line(1,0){.97}}
\put(88.64,34.18){\line(1,0){.97}}
\put(90.58,34.18){\line(1,0){.97}}
\put(92.52,34.18){\line(1,0){.97}}
\put(94.46,34.18){\line(1,0){.97}}
\put(70.93,24.18){\line(1,0){.9688}}
\put(72.867,24.159){\line(1,0){.9688}}
\put(74.805,24.138){\line(1,0){.9688}}
\put(76.742,24.117){\line(1,0){.9688}}
\put(78.68,24.096){\line(1,0){.9688}}
\put(80.617,24.076){\line(1,0){.9688}}
\put(82.555,24.055){\line(1,0){.9688}}
\put(84.492,24.034){\line(1,0){.9688}}
\put(86.43,24.013){\line(1,0){.9688}}
\put(88.367,23.992){\line(1,0){.9688}}
\put(90.305,23.971){\line(1,0){.9688}}
\put(92.242,23.951){\line(1,0){.9688}}
\put(70.93,19.18){\line(1,0){.9674}}
\put(72.864,19.136){\line(1,0){.9674}}
\put(74.799,19.093){\line(1,0){.9674}}
\put(76.734,19.049){\line(1,0){.9674}}
\put(78.669,19.006){\line(1,0){.9674}}
\put(80.604,18.962){\line(1,0){.9674}}
\put(82.538,18.919){\line(1,0){.9674}}
\put(84.473,18.875){\line(1,0){.9674}}
\put(86.408,18.832){\line(1,0){.9674}}
\put(88.343,18.788){\line(1,0){.9674}}
\put(90.278,18.745){\line(1,0){.9674}}
\put(92.212,18.701){\line(1,0){.9674}}
\put(71.43,14.43){\line(1,0){.9688}}
\put(73.367,14.43){\line(1,0){.9688}}
\put(75.305,14.43){\line(1,0){.9688}}
\put(77.242,14.43){\line(1,0){.9688}}
\put(79.18,14.43){\line(1,0){.9688}}
\put(81.117,14.43){\line(1,0){.9688}}
\put(83.055,14.43){\line(1,0){.9688}}
\put(84.992,14.43){\line(1,0){.9688}}
\put(86.93,14.43){\line(1,0){.9688}}
\put(88.867,14.43){\line(1,0){.9688}}
\put(90.805,14.43){\line(1,0){.9688}}
\put(92.742,14.43){\line(1,0){.9688}}
\put(54,-6){\makebox(0,0)[cc]{Figure $5$}}
\multiput(69,32.75)(.0336826347,-.0583832335){334}{\line(0,-1){.0583832335}}
\multiput(69.25,34.25)(.0336538462,-.0556318681){364}{\line(0,-1){.0556318681}}
\multiput(89.75,34.25)(.0337078652,-.0767790262){267}{\line(0,-1){.0767790262}}
\multiput(88.25,34.5)(.0337171053,-.0740131579){304}{\line(0,-1){.0740131579}}
\multiput(20,14.5)(.1420863309,.0337230216){556}{\line(1,0){.1420863309}}
\multiput(20.25,13)(.1387915937,.0337127846){571}{\line(1,0){.1387915937}}
\put(16.25,34.25){\makebox(0,0)[cc]{\tiny{$G_2$}}}
\put(40.75,35.5){\makebox(0,0)[cc]{\tiny{$G_4$}}}
\put(65.5,34.25){\makebox(0,0)[cc]{\tiny{$G_6$}}}
\put(102.5,34.25){\makebox(0,0)[cc]{\tiny{$G_{16}$}}}
\put(17.5,24.25){\makebox(0,0)[cc]{\tiny{$E_1$}}}
\put(29.25,24.75){\makebox(0,0)[cc]{\tiny{$E_2$}}}
\put(42,25){\makebox(0,0)[cc]{\tiny{$E_3$}}}
\put(50.75,27){\makebox(0,0)[cc]{\tiny{$E_4$}}}
\put(81,30.25){\makebox(0,0)[cc]{\tiny{$E_{16}$}}}
\put(8.75,6.5){\line(0,1){36.5}} \put(34.25,6.5){\line(0,1){36.5}}
\put(58.5,6.5){\line(0,1){36.5}} \put(96.25,6.5){\line(0,1){36.5}}
\put(9,43.25){\line(1,0){13}} \put(34.5,43.25){\line(1,0){13}}
\put(58.75,43.25){\line(1,0){13}} \put(96.5,43.25){\line(1,0){13}}
\put(22,43.25){\line(0,-1){36.75}}
\put(47.5,43.25){\line(0,-1){36.75}}
\put(71.75,43.25){\line(0,-1){36.75}}
\put(109.5,43.25){\line(0,-1){36.75}} \put(22,6.5){\line(0,1){0}}
\put(21.75,6.5){\line(-1,0){13}} \put(47.25,6.5){\line(-1,0){13}}
\put(71.5,6.5){\line(-1,0){13}} \put(109.25,6.5){\line(-1,0){13}}
\put(14.5,46.25){\makebox(0,0)[cc]{\tiny{$H_1$}}}
\put(40,46.75){\makebox(0,0)[cc]{\tiny{$H_3$}}}
\put(64.75,46.5){\makebox(0,0)[cc]{\tiny{$H_5$}}}
\put(102.75,47.75){\makebox(0,0)[cc]{\tiny{$H_{15}$}}}
\put(16.5,12.5){\makebox(0,0)[cc]{\tiny{$G_1$}}}
\put(40.75,12.5){\makebox(0,0)[cc]{\tiny{$G_3$}}}
\put(65.25,12.5){\makebox(0,0)[cc]{\tiny{$G_5$}}}
\put(103,12.25){\makebox(0,0)[cc]{\tiny{$G_{15}$}}}
\end{picture}
\end{center}~\\

\textbf{Part II} :  Now we prove the result by induction on $m$.
$Q_7 - D'$ is edge-bipancyclic shows our statement is true for $m
= 3$ (that is $n = 2^3 - 1 = 7$). Suppose the statement is true
for $m = r \geq 3$. Let $n = 2^r - 1$, therefore, $Q_n - D'$ is
edge-bipancyclic where $D'$ denotes any linear PID set or its coset in $Q_n$. \\
We want to show that the statement is true for $m = r+1$. Let $k =
2^{r+1} -1= n+(n+1)$.\\

We decompose $Q_k$ as $Q_k = Q_{n} \times Q_{n+1}$. Let $C =
t_1t_2t_3 \ldots t_{2^{n+1}}t_1$ be a Hamiltonian cycle in $Q_{n+1}$.
Without loss of generality say $t_j$ is of odd parity, for odd
integer $j$, $1 \leq j \leq 2^{n+1}$. Obviously, $t_j$ is of even
parity, for even integer $j$, $1 \leq j \leq 2^{n+1}$.\\

Let $D'_0$ be the Kernel of matrix $H'$
and let $D'_i = D'_0 + e_i$ (for $1 \leq i \leq n$, $e_i =
(0,\ldots,1,\ldots,0)$ $1$ is at the $i^{th}$ position) be its cosets. Now by
lemma \ref{l9} and observation \ref{ob1} , we get the kernel of $H$ say $D$ such that
for odd integer $j$, $D \cap (Q_n, {t_j}) = \phi$ and for even
integer $j$, $D \cap (Q_n, {t_j}) =(D', {t_j})$ where $D' \in
\{D'_0,\ldots,D'_n\}$ ($1 \leq j \leq 2^{n+1}$).\\

We claim that $Q_k - D$ is edge-bipancyclic.\\

Let us denote by $G_j =(Q_n, {t_j}) - ( D \cap (Q_n, {t_j}) )$ ($1
\leq j \leq 2^{n+1}$). For every even integer $j$ ($1 \leq j \leq
2^{n+1}$), $G_j$ is isomorphic to one of the $G'_i$ ($G'_i = Q_n -
D'_i$) and by induction $G'_i$ is an edge-bipancyclic graph ($0
\leq i \leq n$). For odd integer $j$ ($1 \leq j \leq 2^{n+1}$),
$G_j = (Q_n, {t_j})$ is isomorphic to $Q_n$.

Let us denote by $E_j$ the set of cross edges between $G_j$ and
$G_{j+1}$ ($1 \leq j \leq {2^{n+1} - 1}$), therefore $E_j = \{e :$
one end vertex of an edge $e$ is in  $G_l, l$ is an even integer
with $l \in \{j, j+1\} \}$.

By $(b)$ part of Lemma \ref{lm4}, we get $H_j = G_j \bigcup E_j
\bigcup G_{j+1}$ ($j$ is an odd integer, $1 \leq j \leq 2^{n+1}$),
is an edge bipancyclic graph.

As $V(Q_k - D) = \bigcup^{2^{n+1}}_{j=1} V(G_j)$. It is sufficient
to prove that $H_1 \bigcup E_2 \bigcup H_3$ is edge-bipancyclic.

By using Lemma \ref{lm6}, we observe that every Hamiltonian cycle in
$H_1$ contains at least two edges of $G_2$. And by using Lemma
\ref{lm7}, we observe that every Hamiltonian cycle $C$ in $H_3$
contains at least two edges of $G_3$ which corresponds to the
edges of $G_2$.

Now by using Lemma \ref{lm5}, we observe that $H_1 \bigcup E_2
\bigcup H_3$ is edge-bipancyclic.

Thus we proved that $Q_k - D$ is edge-bipancyclic for any linear
perfect independent
dominating set or its coset $D$ in $Q_k$.\\
\end{proof}

\section{Connectivity with Distant Faulty Vertices}

In this section, we study the connectivity of hypercubes by removing a restricted set of faulty vertices say $F$ such that each pair of faulty vertices is distant. Here 'distant' or 'strongly independent' means that the distance between two vertices $u$ and $v$ is greater than or equal to $3$, $d(u,v)\geq 3$. Clearly, each non-faulty vertex is adjacent to at most one faulty vertex. It is proved that (see\cite{gs}) the set $F$ can have up to $2^{n-1}$ vertices in $Q_n$.\\  
Here, we prove that
for strongly independent set $F$ of $Q_n$, $Q_n - F$ is
$(n-1)$-connected for $n \geq 3$.

It is important to see that, which vertex faulty hypercubes maintain
the property like connectivity. In the following theorem, we prove
the existence of such vertex faulty hypercubes.

\begin{thm}\label{tm3} For $n \geq 3$, let $F$ be a strongly independent set in $Q_n$ then $Q_n - F$ is
	$(n-1)$-connected.\end{thm}

\begin{proof} We prove the theorem by induction on $n$. \\
	In $Q_3$, a strongly independent set $F$ is a single vertex in $Q_3$ or
	$F = \{(000),(111)\}$ or $F = \{(100),(011)\}$ or $F =
	\{(010),(101)\}$ or $F = \{(001),(110)\}$. We can easily observe
	that $Q_3 - F$ is $2-$connected graph, see Figure $1$. Thus result holds for $n=3$.\\
	Suppose $n \geq 4$. We assume the result for $n-1$. Consider the subgraphs $(Q_{n-1}, {0})$ and $(Q_{n-1}, {1})$ of $Q_n$ and
	the set $R$ of cross edges between them. So, $Q_n = (Q_{n-1}, {0})\cup
	(Q_{n-1}, {1})\cup R$. Let $F$ be a strongly independent set of $Q_n$ and let $G= Q_n - F$. Let us denote
	by $F_0$ and $F_1$ the strongly independent sets $V(Q_{n-1}, {0})
	\bigcap F$ and $V(Q_{n-1}, {1}) \bigcap F$ respectively. It is easy to observe that $G= G_0 \cup G_1 \cup R'$ where $G_0 = (Q_{n-1}, {0}) - F_0, G_1 = (Q_{n-1}, {1}) - F_1 $ and $R' \subset R$ such that
	 $R' = \{ \langle u, v \rangle = uv : u \in V(G_0)$ and $v \in V(G_1) \} $ 
	Without loss let $|V(G_0)|=k \leq |V(G_1)|=t $ (observe that $2^{n-2} \leq k \leq t$), we can label the vertices of $G_0$ using the set $\{u_1, u_2, u_3,.....,u_k\}$ and $G_1$ using the set $\{v_1, v_2, v_3,.....,v_t\}$, in such a way that $R'= u_1v_1, u_2v_2,u_3v_3,....u_mv_m$, where $2^{n-2} \leq m \leq k$.\\
	By induction hypothesis,
	$G_0 $ and $G_1$ are $(n-2)-$connected.\\
	Claim :  $G = Q_n - F$ is $(n-1)-$connected.\\
	Let $S \subset V(G)$ with $|S| = n-2$. It suffices to prove that $G-S$ is connected.\\
There are at least $2^{n-2} - |S| = 2^{n-2} - (n-2) \geq 1$ edges from the set $R'$ which join $G_0$ to $G_1$ in $G-S$   
Suppose that $S \subset V(G_0)$ or $S \subset V(G_1)$. Without loss, we assume $S \subset V(G_0)$. If $G_0 - S$ is connected or each component of $G_0 - S$ has a neighbor in $G_1$, then $G-S$ is connected because $G_1$ is connected. Assume that $G_0 - S = (Q_{n-1}, {0}) - F_0 - S$ is disconnected and has a component $W$ having no neighbor in $G_1$. As $(Q_{n-1}, {0})$ is $(n-1)-$connected, $(Q_{n-1}, {0}) - S$ is connected. Therefore, $W$ contains at least one vertex say $u_i$ which is adjacent to vertex from $F_0$ only, say $u_{\sigma({i})}$. As $u_{\sigma({i})} \in F_0$ therefore corresponding vertex of $u_i$ from $(Q_{n-1}, {1})$ say $v_i \notin F_1$ so $v_i \in V(G_1)$ . Hence, $W$ is joined to the connected graph $G_1$ by an edge $\langle u_i, v_i \rangle \in R'$.\\
Suppose, instead that $S$ intersect both $V(G_0)$ and $V(G_1)$. Let $S_0 = V(G_0) \bigcap S$ and $S_1 = V(G_1) \bigcap S$. Then $S = S_0 \cup S_1$ and $ S_0 \cap S_1 \neq \phi$. Since $|S| \leq {n-2}$, $|S_0| \leq {n-3}$ and $|S_1| \leq {n-3}$. Hence, both $G_0 - S_0$ and $G_1 - S_1$ are connected because $G_0$ and $G_1$ are $(n-2)$-connected. Further $G_0 - S_0$ and $G_1 - S_1$ are joined to each other by at least one edge of $R'$ in $G-S$. Therefore, $G-S$ is connected. Thus $G$ is $(n-1)-$connected.               
\end{proof}

\section{Connectivity and Bipancyclicity with Faulty Vertices in Dominating Set.}

 We are interested in perfect dominating sets because
 their removal results in a subgraph in which each vertex has degree one less than in the original graph.
 In this section, we prove the existence of perfect dominating set $D$ of a
hypercube $Q_n$($n \geq 3$) such that $Q_n - D$ is
$(n-1)$-regular, $(n-1)$-connected and bipancyclic.\\
By using Theorem \ref{tm2} and Theorem \ref{tm3}, we obtain the following
Theorem which states that the Hamming shell is regular, connected
and edge-bipancyclic.

\begin{thm}\label{tm4}
 For $n = 2^k -1$, $k \geq 2$, let $D$ denote a perfect independent dominating set in $Q_n$. Then
$Q_n - D$ is $(n-1)$-regular and $(n-1)$-connected and in addition,
if $D$ is any coset of any linear perfect independent dominating
set of $Q_n$ then $Q_n - D$ Hamiltonian for $k = 2$ and
edge-bipancyclic for $k \geq 3$.
 Moreover, this is the smallest set satisfying the above property.
 \end{thm}
\begin{proof} A perfect independent
dominating set $D$ in $Q_n$ for $n = 2^k -1$,
 $k \geq 2$ is a strongly independent set in $Q_n$ and hence $Q_n - D$ is $(n-1)$-regular where $|D|= 2^{n-k}$.
 Also by Theorem \ref{tm3}, $Q_n - D$ is $(n-1)$-connected.\\
If $D$ is any coset of any linear perfect independent dominating
set of $Q_n$ then $Q_n - D$ is Hamiltonian for $k = 2$ and
edge-bipancyclic for $k \geq 3$ (by Theorem \ref{tm2}). Thus we proved
the existence of perfect independent dominating set $D$ of $Q_n$
such that $Q_n - D$ is $(n-1)$-regular, $(n-1)$-connected and
Hamiltonian for $k = 2$ and edge-bipancyclic for $k \geq 3$.

 In case if we choose a set of vertices in $Q_n$ say $F_v$
 such that $|F_v|= 2^{n-k}-1$ then because of $n$-regularity of
 $Q_n$ exactly $(n+1)$ vertices remain unattended which shows $Q_n - F_v$ is non-regular, therefore $D$ is the
  smallest set satisfying the property above.\end{proof}

We need following lemmas before proving the main Theorem of this
section.

\begin{lem}[\cite{bd}\label{l1}] If $G$ is an
 $n$-regular graph, then a perfect dominating set of $G$ can be
 described as follows : If $H$ is an $(n-1)$-regular induced
 subgraph of $G$, then $V(G) - V(H)$ is a perfect dominating set of
 $G$.\end{lem}

 \begin{lem} [\cite{mw}\label{l2}] Let for $k \geq 2$, $P_k$ be denoted the path of length $k-1$.
 Then the mesh $P_m \times P_n$ ($m,n \geq 2$) contains a cycle of length $l$ for any even integer $l$ with $4 \leq l \leq mn$.\end{lem}

\begin{lem}[\cite{sp}\label{l3}] For any connected graphs, $G$
and $H$, $\kappa(G \times H) \geq \kappa(G) + \kappa(H)$
($\kappa(G)$ denotes connectivity of $G$).\end{lem}

Now we prove existence of a total perfect dominating set of
hypercubes removal of which gives a regular, connected and
bipancyclic subgraph.

\begin{thm} In an $n$-dimensional hypercube $Q_n$ ($n \geq 4$) there
exists a total perfect dominating set $D$ of $Q_n$ for $ 2^k - 1 <
n < 2^{k+1} -1$ and $k \geq 2$ with $|D|=2^{(n-k)}$ such that $Q_n
- D$ is $(n-1)-$regular, $(n-1)-$connected and
bipancyclic.\end{thm}

\begin{proof}For $n \geq 4$ consider $k \geq 2$ such that $ 2^k - 1 < n
< 2^{k+1} -1$. Now we write $Q_n = Q_m \times Q_{n-m}$ where $m =
2^k - 1$. For any $t \in V(Q_{n-m})$ we denote by $(Q_{m}, {t})$
 the subgraph of $Q_n$ induced by the vertices whose last $(n-m)$ components form the tuple $t$.

By Theorem \ref{tm2}, for every $D$ which is any coset of any linear
perfect independent dominating set of $Q_m$, $H = Q_m - D$ is
$(m-1)$-regular, $(m-1)$-connected and, in addition, it is
Hamiltonian for $k = 2$, and
edge-bipancyclic for $k \geq 3$.\\
Claim : The subgraph $G =  H \times Q_{n-m}$ of $Q_n$ is $(n-1)$-regular, $(n-1)$-connected and bipancyclic.\\
For $s \in V(H)$ and $t \in V(Q_{n-m})$ we denote by $(s,t)$ the
vertex in $Q_n$ whose first $m$ components form the tuple $s$ and
whose last $(n-m)$ components form the tuple $t$. Thus we have
$V(G)= V(H)\times V(Q_{n-m}) = \{(s,t): s \in V(H), t \in
V(Q_{n-m})\}$. 

Firstly, we prove that $G$ is $(n-1)-$regular.\\
As degree of $s$ that is $d(s) = m-1$ in $H$ and $d(t) = n-m$ in
$Q_{n-m}$,
we have $d(s,t) = (m-1) + (n-m)  = n-1$ in $G$, for every vertex in $G$.

Now, we will prove that $G$ is $(n-1)-$connected.\\
 We know that $\kappa(G) \leq \delta(G) = n-1$.
Also by Lemma \ref{l3}, $\kappa(G)= \kappa(H \times Q_{n-m}) \geq \kappa(H) + \kappa(Q_{n-m}) = (m-1) + (n-m)$ gives us $\kappa(G)= n-1$.

Lastly, we will prove that $G$ is bipancyclic for $k \geq 3$.\\
 As $G = H \times Q_{n-m}$ contains the mesh $P_{m{(2^{m-k})}} \times P_{2^{n-m}}$ as a spanning subgraph,
  bipancyclicity of $G$ follows by the Lemma \ref{l2}.\\
  For $t \in
 V(Q_{n-m})$, let us denote by $D' = \bigcup (D, t)$ where $V(D, t) = \{(s,t): s \in V(D), t \in
V(Q_{n-m})\}$. We note that $G = Q_n - D'$.

By Lemma \ref{l1}, $D'$ is a perfect dominating set of $Q_n$ and in
addition, $D'$ is total by definition. Thus $D'$ is a total
perfect dominating set of $Q_n$ removal of which gives an induced
subgraph $G$ which is $(n-1)$-regular, $(n-1)$-connected and
bipancyclic.\\

This completes the proof.
\end{proof}

 \noindent \textbf{Concluding remarks}\\
Thus we proved that for any coset of any linear perfect
independent dominating set $D$ of $Q_n$, $Q_n - D$ is
edge-bipancyclic. Whether the same property is true for any
perfect
independent dominating set of $Q_n$?~\\

 \noindent {\bf Acknowledgment:} The author gratefully
 acknowledges the Department of Science and Technology, New Delhi, India
 for the award of Women Scientist Scheme (SR/WOS-A/PM-79/2016) for research in Basic/Applied Sciences.

{\centerline{************}}

\bibliographystyle{amsplain}

\end{document}